\documentclass[10pt, oneside]{amsart}

\usepackage{amsmath, amsmath, amsthm, amsfonts}
\usepackage{amssymb}
\usepackage{mathrsfs}
\usepackage{eucal}
\usepackage[all]{xy}
\usepackage{float}
\usepackage{stmaryrd}
\usepackage{hyperref}
\usepackage{todonotes}
\usepackage{enumitem}
\usepackage{eucal}
\usepackage[textwidth=14.5cm, textheight=22cm]{geometry}
\usepackage{setspace}
\usepackage{color}

\hypersetup{%
  bookmarksnumbered=true,%
  colorlinks=true,%
  linkcolor=blue,%
  citecolor=blue,%
  filecolor=blue,%
  menucolor=blue,%
  urlcolor=blue,%
  pdfnewwindow=true,%
  pdfstartview=FitBH}

\let\OLDthebibliography\thebibliography
\renewcommand\thebibliography[1]{
 \OLDthebibliography{#1}
 \setlength{\parskip}{1.4pt}
 \setlength{\itemsep}{0pt plus 0.2ex}
}

\newtheorem{theorem}{Theorem}[section]
\newtheorem{proposition}[theorem]{Proposition}
\newtheorem{corollary}[theorem]{Corollary}
\newtheorem{lemma}[theorem]{Lemma}

\newtheorem{definition}[theorem]{Definition}
\newtheorem{example}[theorem]{Example}

\newtheorem{remark}[theorem]{Remark}

\newcommand{\Mdef}[2]{\newcommand{#1}{{#2}}}

\Mdef{\bA}{\mathbb{A}}
\Mdef{\bB}{\mathbb{B}}
\Mdef{\bC}{\mathbb{C}}
\Mdef{\bD}{\mathbb{D}}
\Mdef{\bE}{\mathbb{E}}
\Mdef{\bF}{\mathbb{F}}
\Mdef{\bG}{\mathbb{G}}
\Mdef{\bH}{\mathbb{H}}
\Mdef{\bI}{\mathbb{I}}
\Mdef{\bJ}{\mathbb{J}}
\Mdef{\bK}{\mathbb{K}}
\Mdef{\bL}{\mathbb{L}}
\Mdef{\bM}{\mathbb{M}}
\Mdef{\bN}{\mathbb{N}}
\Mdef{\bO}{\mathbb{O}}
\Mdef{\bP}{\mathbb{P}}
\Mdef{\bQ}{\mathbb{Q}}
\Mdef{\bR}{\mathbb{R}}
\Mdef{\bS}{\mathbb{S}}
\Mdef{\bT}{\mathbb{T}}
\Mdef{\bU}{\mathbb{U}}
\Mdef{\bV}{\mathbb{V}}
\Mdef{\bW}{\mathbb{W}}
\Mdef{\bX}{\mathbb{X}}
\Mdef{\bY}{\mathbb{Y}}
\Mdef{\bZ}{\mathbb{Z}}
\Mdef{\cL}{\mathcal{L}}
\Mdef{\cF}{\mathcal{F}}

\Mdef{\mcA}{\mathcal{A}}
\Mdef{\mcB}{\mathcal{B}}
\Mdef{\mcC}{\mathcal{C}}
\Mdef{\mcD}{\mathcal{D}} 
\Mdef{\mcE}{\mathcal{E}}
\Mdef{\mcF}{\mathcal{F}}
\Mdef{\mcG}{\mathcal{G}}
\Mdef{\mcH}{\mathcal{H}} 
\Mdef{\mcI}{\mathcal{I}}
\Mdef{\mcJ}{\mathcal{J}}
\Mdef{\mcK}{\mathcal{K}}
\Mdef{\mcL}{\mathcal{L}}

\Mdef{\mcM}{\mathcal{M}}
\Mdef{\mcN}{\mathcal{N}}
\Mdef{\mcO}{\mathcal{O}}
\Mdef{\mcP}{\mathcal{P}}
\Mdef{\mcQ}{\mathcal{Q}}
\Mdef{\mcR}{\mathcal{R}}
\Mdef{\mcS}{\mathcal{S}}
\Mdef{\mcT}{\mathcal{T}}
\Mdef{\mcU}{\mathcal{U}}
\Mdef{\mcV}{\mathcal{V}}
\Mdef{\mcW}{\mathcal{W}}
\Mdef{\mcX}{\mathcal{X}}
\Mdef{\mcY}{\mathcal{Y}}
\Mdef{\mcZ}{\mathcal{Z}}

\Mdef{\At}{\tilde{A}}
\Mdef{\Bt}{\tilde{B}}
\Mdef{\Ct}{\tilde{C}}
\Mdef{\Et}{\tilde{E}}
\Mdef{\Ht}{\tilde{H}}
\Mdef{\Kt}{\tilde{K}}
\Mdef{\Lt}{\tilde{L}}
\Mdef{\Mt}{\tilde{M}}
\Mdef{\Nt}{\tilde{N}}
\Mdef{\Pt}{\tilde{P}}

\def\endash{\mathchar"707B}

\newcommand{\leftmod}{\endash \textnormal{mod}}

\newcommand{\Einfty}{\textnormal{E}_\infty}

\newcommand{\einf}{\textnormal{E}_\infty}
\newcommand{\eginf}{\textnormal{E}^G_\infty}

\newcommand{\co}{\colon \!}

\newcommand{\ho}{\textnormal{Ho}}
\newcommand{\Ig}{\textnormal{Ig}}

\newcommand{\smashprod}{\wedge}

\newcommand{\tensor}{\otimes}

\newcommand{\adjunction}[4]{
\xymatrix{
#1:#2 \ar@<0.7ex>[r] &
\ar@<0.7ex>[l] #3:#4
}}

\newcommand{\lra}{\longrightarrow}
\newcommand{\spO}{\textnormal{Sp}}
\newcommand{\spS}{\textnormal{Sp}^\Sigma}

\Mdef{\infl}{\mathrm{inf}}
\Mdef{\defl}{\mathrm{def}}
\Mdef{\res}{\mathrm{res}}
\Mdef{\ind}{\mathrm{ind}}
\Mdef{\coind}{\mathrm{coind}}
\Mdef{\Top}{\mathsf{Top}}
\Mdef{\sset}{\mathsf{sSet}}

\Mdef{\Comm}{\mathsf{Comm}}

\newcommand{\algin}[2]{#1 \endash \textnormal{alg-in}\endash #2}
\newcommand{\mcOalgbbC}{\algin{\mcO}{\bC}}
\newcommand{\mcOalgbbD}{\algin{\mcO}{\bD}}

\newcommand{\st}{\; |\;}

\newcommand{\Ch}{\mathsf{Ch}}

\definecolor{darkgreen}{RGB}{85,107,47}

\newcommand{\GSp}{G \endash \spO}
\newcommand{\NSp}{N \endash \spO}
\newcommand{\SnSp}{\Sn \endash \spO}
\newcommand{\GSn}{G\times \Sigma_n}
\newcommand{\GSnSp}{(G\times \Sigma_n) \endash \spO}
\newcommand{\Sn}{\Sigma_n}
\newcommand{\FGSn}{\mathcal{F}_G(\Sn)}
\newcommand{\FSn}{\mathcal{F}_1(\Sn)}
\newcommand{\GI}{\Ig}

\newcommand{\sm}{\wedge}
\newcommand{\Ninfty}{N_{\infty}}

\title{An algebraic model for rational naive-commutative equivariant ring spectra}
\author[Barnes]{David Barnes}
\address[Barnes]{Mathematical Sciences Research Centre, Queen's University Belfast}
\email{d.barnes@qub.ac.uk}
\author[Greenlees]{J.P.C. Greenlees}
\address[Greenlees]{School of Mathematics and Statistics, Hicks
  Building, Sheffield, S3 7RH, UK}
\email{j.greenlees@sheffield.ac.uk}
\author[K\c{e}dziorek]{Magdalena K\c{e}dziorek}
\address[K\c{e}dziorek]{SV UPHESS BMI, \'Ecole Polytechnique F\'ed\'erale de Lausanne}
\email{magdalena.kedziorek@epfl.ch}

\begin{document}

\begin{abstract}
Equipping a non-equivariant topological $E_\infty$ operad with the trivial $G$--action gives 
an operad in $G$--spaces. The algebra structure encoded by this operad in $G$--spectra is characterised homotopically by 
having no non-trivial multiplicative norms. Algebras over this operad are called 
na\"ive--commutative ring $G$--spectra.  
In this paper we let $G$ be a finite group and we
show that commutative algebras in 
the algebraic model for rational $G$--spectra model 
the rational na\"ive--commutative ring $G$--spectra. 
\end{abstract}

\maketitle

\tableofcontents

\section{Introduction}

We are interested in the category of $G$--spectra, where $G$ is a finite group and 
the indexing universe (unless specifically stated otherwise) 
is a complete $G$--universe $U$. 
This is a model for the equivariant stable homotopy category 
where all $G$--representation spheres are invertible. 
Building on work of Greenlees and May \cite[Appendix A]{gremay95} and Barnes \cite{barnesfinite}, K\c{e}dziorek \cite{KedziorekExceptional} gave a symmetric monoidal algebraic model for the category of rational $G$--spectra, when $G$ is a finite group. 
As a consequence, one obtains a model for rational ring $G$--spectra in terms of 
monoids in the algebraic model. However, this does not imply results
about strict commutative rational ring $G$--spectra ($\Comm$-algebras). This is because of 
the well-known but surprising result that \emph{symmetric} 
monoidal Quillen functors can fail to preserve \emph{commutative} 
monoids (algebras for an operad $\Comm$) in the equivariant setting.

Recent work of Blumberg and Hill \cite{BlumbergHillNorms}
describes a class of commutative multiplicative structures on the
equivariant stable homotopy category. These multiplicative structures are governed by $G$--operads called $\Ninfty$--operads. Roughly speaking, a multiplicative structure is characterised by the set of Hill-Hopkins-Ravenel norms which exist on the commutative algebras corresponding to it.

This work of Blumberg and Hill raises a question: which level of commutative ring $G$--spectra is modelled by commutative algebras in the algebraic model of \cite{KedziorekExceptional}?

\subsubsection*{Contribution of this paper}
In this paper we show that commutative algebras in the algebraic model for rational $G$--spectra are Quillen equivalent to rational na\"ive-commutative ring $G$--spectra, that is, the algebras
over the \emph{non-equivariant} $E_\infty$ operad equipped with the trivial $G$--action, 
denoted~$E_\infty^1$.
In the view of \cite{BlumbergHillNorms} this operad has the least commutative structure,
in particular it encodes the algebra structure without multiplicative norms.

Furthermore, the Quillen equivalences between the algebraic model and rational $G$-spectra do
not lift to Quillen equivalences of any more structured commutative objects, such as  
$\Comm$-algebras or algebras for a different $\Ninfty$--operad. 
Thus we have an interesting example of where 
symmetric monoidal Quillen equivalences do not induce 
Quillen equivalences of categories of $\Comm$-algebras. 

There are clear reasons why it should be expected that the
algebraic model only captures the lowest level of commutativity in the equivariant setting.
The methods used to obtain the monoidal algebraic model for
rational $G$--spectra rely on the complete idempotent splitting using methods of
Barnes \cite{barnessplitting} and left Bousfield localisation techniques.
The splitting operation cannot preserve norms, and hence won't preserve more structured commutative objects (see for example McClure \cite{McClureTate} or Section \ref{sec:obstructions}).
Equally, the operad $\Comm$ in the algebraic model doesn't encode any
additional structure beyond that of commutative algebras in the usual
sense,  and thus the category of its algebras can only model the 
rational na\"ive--commutative ring $G$--spectra.

The new results of the paper are in the analysis of the relation between a left Bousfield localisation at an object $E$ of $G$--spectra (denoted by $L_E \GSp$) and an $N_\infty$--operad $\mcO$.
In Section \ref{section:liftingModels} we identify conditions on $\mcO$ and $E$ which ensure that
$\algin{\mcO}{L_E \GSp}$ has a right-lifted model structure from $L_E \GSp$. 
In particular, the fibrant replacement functor of this lifted model structure on $\algin{\mcO}{L_E \GSp}$
is a functorial $E$-localisation that preserves $\mcO$--algebras. 

In the appendix we construct a model structure on $\algin{\Comm}{\GSp_\bQ}$, where $G$ is \emph{any compact Lie group}. This is a model structure on commutative ring $G$--spectra where weak equivalences are rational equivalences. 
This procedure does not follow from the previous analysis, since $\Comm$ is not an $N_\infty$--operad, but instead uses the existence of a commutative multiplication on the rational sphere spectrum.
The authors included this appendix as they were unable to find a sufficiently detailed reference for this result, 
which is needed in Section \ref{sec:isotropic}. 

This paper is also the first step in understanding what level of commutativity is modelled by the algebraic model for rational $SO(2)$--spectra of \cite{BGKSso2}. In this case various left Bousfield localisations and idempotent splittings are also used to obtain the algebraic model. Understanding the interplay between left Bousfield localisations, idempotents 
and operads $\mcO$ is crucial in this more complicated setting.
This in turn will be used in \cite{NComSO2} to show that rational $SO(2)$--equivariant
elliptic cohomology constructed in \cite{gre99} is modelled by
a rational $E_\infty^1$-ring-$G$-spectrum and hence there is a
description of the category of modules over this ring spectrum in
terms of algebraic geometry.

\subsubsection*{Relation to other work} There are recent results by 
Hill and Hopkins \cite{HillHopkins16}, Guti\'errez and White \cite{gw17}, and Hill \cite{hill17} on the interaction of Bousfield localisations (understood as a fibrant replacement in the left Bousfield localised model structure) with the structure of $\mcO$--algebras, for various $N_\infty$--operads $\mcO$. 
Our approach differs from the references mentioned above as we are 
primarily interested in creating a version of $E$-local model structure on 
$\mcO$--algebras in $G$--spectra.
This gives a stronger result than simply looking at images of $\mcO$--algebras under localisation functors. 
The existence of such a model structure implies that one can preserve the commutative structure 
modelled by $\mcO$ while doing a fibrant replacement relative to $E$. 
In particular, the conditions of Section \ref{section:liftingModels} 
are much more direct than those of
\cite[Theorems 6.2 and 6.3]{HillHopkins16}.

\subsubsection*{An algebraic model}

Our primary aim is Theorem \ref{thn:fullcompare}
which states  that for a finite group $G$, there is a zig-zag of Quillen equivalences 
\[
\algin{\einf^1}{  (G\endash \mathrm{Sp})_\bQ   }
\  \simeq \ 
\algin{\Comm} {\prod_{(H) \leqslant G} \Ch(\bQ[W_G H])}
\]
where $W_GH=N_GH/H$ is the Weyl group of $H$ in $G$.
This is based on the chain of symmetric monoidal Quillen equivalences of
\cite[Corollary 5.10]{KedziorekExceptional}.
To illustrate the whole path of that result, we present a diagram which shows every step of this comparison. The reader may wish to refer to this diagram now, but the notation will be introduced as we proceed. Left Quillen functors are placed on the left and we use the notation $N=N_GH$ and  $W=W_GH=N_GH/H$.  Also, $e_{(H)_G}$ denotes the idempotent supported on the conjugacy class of $H$ in the rational Burnside ring of $G$.  We present on the right the series of Quillen equivalence at the level of (na\"ive) commutative ring objects.

\[\begin{array}{lccr}
\xymatrix@R=2pc{
L_{e_{(H)_G}S_{\bQ}}(G\endash \spO)
\ar@<-1ex>[d]_{i^{*}}
\\
L_{e_{(H)_N}S_{\bQ}}(N \endash \spO)
\ar@<-1ex>[u]_{F_N(G_+,-)}
\ar@<+1ex>[d]^{(-)^H}
\\
L_{e_1S_\bQ}(W\endash \spO)
\ar@<+1ex>[u]^{ \epsilon^\ast}
\ar@<+1ex>[d]^{\mathrm{res}}
\\
\spO_\bQ[W]
\ar@<+1ex>[u]^{\mathrm{I}_t^c}
\ar@<+1ex>[d]^{ \mathrm{Sing}\circ \mathbb{U}}
\\
\spO_\bQ^\Sigma[W]
\ar@<+1ex>[u]^{\mathbb{P} \circ |-|}
\ar@<-1ex>[d]_{H\bQ\wedge -}
\\
(H\bQ\leftmod)[W]
\ar@<-1ex>[u]_{U}
\ar@<-0ex>[d]_{\text{zig-zag of}}
\\
\Ch(\bQ[W])
\ar@<-0ex>[u]_{\text{Quillen equivalences}}
}
& \quad \quad \quad &
\xymatrix@R=2pc{
\algin{\einf^1}{  L_{e_{(H)_G}S_{\bQ}}(\GSp)   }
\ar@<-1ex>[d]_{i^{*}}
\\
\algin{\einf^1}{  L_{e_{(H)_N}S_{\bQ}}(\NSp)   }
\ar@<-1ex>[u]_{F_N(G_+,-)}
\ar@<+1ex>[d]^{(-)^H}
\\
\algin{\einf^1}{   L_{e_1S_\bQ}(W\endash \spO)    }
\ar@<+1ex>[u]^{ \epsilon^\ast}
\ar@<+1ex>[d]^{\mathrm{res}}
\\
\algin{\einf^1}{   \spO_\bQ[W]   }
\ar@<+1ex>[u]^{\mathrm{I}_t^c}
\ar@<-1ex>[d]_{\text{change of}}
\\
\algin{\Comm}{   \spO_\bQ[W]    }
\ar@<-1ex>[u]_{\text{operads}}
\ar@<+1ex>[d]^{ \mathrm{Sing}\circ \mathbb{U}}
\\
\algin{\Comm}{   \spS_\bQ[W]    }
\ar@<+1ex>[u]^{\mathbb{P} \circ |-|}
\ar@<-1ex>[d]_{H\bQ\wedge -}
\\
\algin{\Comm}{   (H\bQ\leftmod)[W]  }
\ar@<-1ex>[u]_{U}
\ar@<-0ex>[d]_{\text{zig-zag of}}
\\
\algin{\Comm}{   \Ch(\bQ[W])   }
\ar@<-0ex>[u]_{\text{Quillen equivalences}}
}
\end{array}
\]

\subsubsection*{Outline of the paper}
The aim is to show that the right hand column of the above diagram is 
a series of Quillen equivalences. 
We start by recalling basic definitions and results from Blumberg and Hill \cite{BlumbergHillNorms} on $N_\infty$--operads in Section \ref{sec:recolloectionBH}. 
Section \ref{sec:strategy} describes a general strategy of lifting adjunctions and model structures to the level of algebras over operads. It also provides reasonable conditions which imply that a Quillen equivalence at the level of underlying categories lifts to a Quillen equivalence at the level of algebras over an operad.

In Section \ref{section:liftingModels}, we give a 
general result with directly checkable assumptions that 
ensures there is a right--lifted model structure on 
$\mcO$--algebras in $L_E \GSp$ (and related categories), see 
Theorem \ref{thm:algebramodelstructure}.
The technical heart of the paper is Section \ref{sec:isotropic} where 
the details of Theorem \ref{thm:algebramodelstructure} are proven. 

In Section \ref{sec:liftingQE} we verify that the assumptions of the previous 
results hold in the cases of interest: the zig-zag of Quillen equivalences 
from \cite{KedziorekExceptional}. This section gives the main result, 
of the paper: Theorem \ref{thn:fullcompare}. 

Finally, the appendix shows that there is a model structure on 
\emph{rational} commutative ring $G$--spectra right-lifted from 
the rational positive stable model structure on $G$--spectra.
This uses a different method to the results on $\mcO$--algebras 
for $\Ninfty$--operads $\mcO$ and works for any compact Lie group $G$. 

\subsubsection*{Notation} We stick to the convention of writing a left adjoint in any adjoint pair on top or on the left of a right one.

\subsubsection*{Acknowledgements} The authors would like to thank the Mathematisches Forschungsinstitut Oberwolfach for giving them an ideal venue to work on this project as part of the Research in Pairs Programme.

\section{Recollections on \texorpdfstring{$N_\infty$}{N-infinity}--operads}\label{sec:recolloectionBH}

\begin{definition}\label{def:Goperad}
A $G$-operad $\mcO$ consists of a sequence of $G\times \Sigma_n$--spaces $\mcO_n$, $n\geq 1$, such that
\begin{enumerate}
\item There is a $G$-fixed identity element $1\in \mcO_1$
\item there are $G$-equivariant composition maps
$$\mcO_k\times \mcO_{n_1}\times \mcO_{n_1} ... \times \mcO_{n_k} \lra \mcO_{n_1+...+n_k}$$
satisfying the usual compatibility conditions with each other and with the symmetric group actions, see \cite[2.1]{CostenobleWaner}.
\end{enumerate}
\end{definition}

Recall that for a group $G$ a \emph{family} is a collection of subgroups of $G$ closed under conjugation and taking subgroups. If $\cF$ is a family, then a universal space $E\cF$ for the family $\cF$ is a $G$-space with the property that
$$(E\cF)^H=\begin{cases}
* & H\in \cF\\
\emptyset & H\not \in \cF.
\end{cases}$$

\begin{definition}\label{def:Ninfty_operad}\cite[Definition 1.1]{BlumbergHillNorms} A $G$-operad
$\mcO$ is an $\Ninfty$-operad if
\begin{itemize}
\item $\mcO(n)$ is universal space for a family $\mcF_n(\mcO)$ of
subgroups of $G \times \Sn$ that contains all subgroups of the form
$H \times \{1\}$. 
\item $\Sn$ acts freely on $\mcO(n)$.
\item $\mcO(0)$ is $G$-contractible.
\end{itemize}
\end{definition}

If we forget the $G$ action on an $\Ninfty$ operad 
we obtain an $\einf$ operad in the classical  sense.
The condition that $\Sn$ acts freely on $\mcO(n)$ is precisely the statement that
for any subgroup of the form $\{ e \} \times K$ of $\GSn$, $\mcO(n)^{\{ e \} \times K} = \emptyset$.
It follows that $\mcF_n(\mcO)$ cannot contain any subgroup of the form $\{ e \} \times K$.
Thus any $\Gamma \in \mcF_n(\mcO)$ must have trivial intersection with $\{ e \} \times \Sn$
(or $\Gamma$ would share a non-trivial subgroup with $\{ e \} \times \Sn$, which in turn
would be an element of the family $\mcF_n(\mcO)$).
We say that a subgroup $\Gamma$ of $\GSn$ is \textbf{admissible} if 
$\Gamma \cap \{ e \} \times \Sn=\{e \times e\}$.

Since $\Ninfty$--operads are characterised by the families of subgroups from the definition above, let us fix some notation. We will use the notation $\einf^1$ for an $\Ninfty$ operad, where for every $n$ we take the family 
$\FSn= \{H\times \{ e \} \mid H \leqslant G \}$. Notice that those are the \textbf{minimal} families for an $\Ninfty$ operad with respect to $G$.
On the other hand, we can take the \textbf{maximal} families in the definition above, that is, for $n\geq 1$ take the family  $\FGSn$ of all admissible subgroups of $G\times \Sn$.
We will use the notation $\eginf$ for any operad with $\mcF_n(\mcO) = \FGSn$ for all $n \geqslant 0$.

\begin{example}\label{ex:operads}
Let $U$ denote a $G$-universe, i.e. a countably infinite-dimensional real $G$-inner product space which contains each finite dimensional sub-representation infinitely often. The \emph{linear isometries operad} $\cL(U)$ is a $G$-operad such that $\cL(U)(n):= \cL(U^n,U)$, where $ \cL(U^n,U)$ denotes non-equivariant linear isometries from $U^n$ to $U$. It is a $G\times \Sn$ space by conjugation and diagonal action. The identity map $U\lra U$ is the distinguished element of $\cL(U)(1)$ and the structure maps are given by composition.

If $U$ is a \textbf{complete} $G$-universe then $\cL(U)$ is an example of $\eginf$ operad.

If $U$ is a \textbf{trivial} $G$-universe then $\cL(U)$ is an example of $\einf^1$ operad.

Both are examples of $\Ninfty$ operads.
\end{example}

Since the category of $G$-spectra is tensored over $G$-spaces we can consider $\mcO$-algebras in $G$-spectra, where $\mcO$ is a $G$-operad (an operad in $G$-spaces). In particular this applies when $\mcO$ is an $\Ninfty$-operad.

In the next proposition we consider the category of orthogonal $G$-spectra with the positive stable model structure of \cite[Section III.5]{mm02}.

\begin{proposition}\cite[Proposition A.1]{BlumbergHillNorms} For any $\Ninfty$-operad $\mcO$ such that each $\mcO(n)$ has a homotopy type of a $G\times \Sn$-CW complex there exists a right-lifted model structure on $\mcO$-algebras in  orthogonal $G$--spectra (i.e. the weak equivalences and fibrations are created in the category of orthogonal $G$--spectra with the positive stable model structure).
\end{proposition}

By \cite[Definition 3.9]{BlumbergHillNorms} a map $\mcO \lra \mcO'$ of $G$-operads is a weak equivalence if for any $n\geq 1$ the map $\mcO(n)^\Gamma \lra \mcO'(n)^\Gamma$ is an equivalence for all subgroups $\Gamma \subseteq G\times \Sn$.

With this definition we recall the comparison result for algebras over weakly equivalent $\Ninfty$ operads.

\begin{lemma}\cite[Theorem A.3]{BlumbergHillNorms} Suppose $f: \mcO \lra \mcO'$ is a map of $\Ninfty$ operads,  such that each $\mcO(n), \mcO'(n)$ have homotopy types of $G\times \Sn$-CW complexes and $\mcO(1),\mcO'(1)$ have nondegenerate $G$-fixed basepoints. Then the adjunction $(f_*,f^*)$ is a Quillen equivalence between categories of their algebras.
\end{lemma}

\section{Strategy for lifting model structures}\label{sec:strategy}

We are interested in considering $\mcO$-algebras in various symmetric
monoidal categories  $\bC$ for an operad $\mcO$. We want to put model
structures on these categories of algebras, and to show that Quillen
equivalences on the underlying categories lift to Quillen equivalences
on categories of algebras. The purpose of this section is to record
the formal part of this process, showing the basic conditions that
need to be satisfied in order to perform a right lifting of model structures and Quillen equivalences.
We will apply this strategy numerous times in the rest of the paper.

Suppose $\mcO$ is an operad
in a category $V$ and $\bC$ is tensored over $V$. In that situation
there is an adjunction
$$\adjunction{F_{\mcO}}{\bC}{\algin{\mcO}\bC}{U}$$
where $U$ is the forgetful functor and $F_{\mcO}$ is the free
$\mcO$-algebra functor
$$F_{\mcO}X =\coprod_n \mcO (n)\tensor_{\Sigma_{n}} X^{\tensor n} .$$

If $\bC$ is a model category we want to right lift the model structure from $\bC$ to
the category $\algin{\mcO}{\bC}$ of $\mcO$-algebras in $\bC$ using
the right adjoint $U$, i.e. we want the weak equivalences and fibrations to be created by $U$.

To do that we will use Kan's result for right lifting model structures.

\begin{lemma}\cite[Theorem 11.3.2]{hir03}
\label{lem:liftedmodel} Suppose $\bC$ is a cofibrantly generated model category with a set of generating acyclic cofibrations $J$.  If
\begin{itemize}
\item  $F_{\mcO}$ preserves small objects (or equivalently $U$ preserves filtered
  colimits) and
  \item every transfinite composition of pushouts (cobase extensions)
  of elements of $F_{\mcO}(J)$ is sent to a weak equivalence
  of $\bC$ by $U$,
\end{itemize}
then the model structure on $\bC$ may be lifted using the right
adjoint to give a cofibrantly generated model structure on
$\algin{\mcO}{\bC}$. The functor $U$ then creates fibrations and weak
equivalences.
\end{lemma}

Since later on we will be interested in lifting Quillen equivalences we consider the following situation.
Suppose there are two symmetric monoidal categories $\bC$ and $\bD$, both tensored over a closed symmetric monoidal category $\bV$. Let $\mcO$ be an operad in $\bV$. Suppose further that we have an adjunction
\[
\xymatrix@C=3pc{
\bC \ar@<0.5ex>[r]^L &\bD. \ar@<0.5ex>[l]^R
}
\]
If both categories $\bC, \bD$ have model structures which can be lifted to $\mcO$ algebras we obtain the following diagram
\[
\xymatrix{
\algin{\mcO}{\bC} \ar@<0.5ex>[d]^U &
\algin{\mcO}{\bD} \ar@<0.5ex>[d]^U \\
\bC \ar@<0.5ex>[r]^L \ar@<0.5ex>[u]^{F_{\mcO}}&\bD \ar@<0.5ex>[l]^R \ar@<0.5ex>[u]^{F_{\mcO}}
}
\]
It is natural to ask under what assumptions on $L$ and $R$ the adjunction lifts to the adjunction at the level of algebras and moreover when is it a Quillen pair and a Quillen equivalence. We will answer these questions below.

\begin{remark} Notice that we can consider a generalisation of a situation above, namely algebras over different operads in $\bC$ and in $\bD$. Since we won't use that generalisation in this paper, we will stick to the simpler case of algebras over the same operad both in $\bC$ and in $\bD$.
\end{remark}

For the rest of this section, 
we assume that the hypotheses of Lemma \ref{lem:liftedmodel} have been verified for
both $(\bC , \mcO)$ and $(\bD  , \mcO)$, so that $\algin{\mcO}{\bC}$ and
$\algin{\mcO}{\bD}$  have both been given the lifted model structures.

Now suppose that the left adjoint $L$ is tensored over a closed symmetric monoidal category $\bV$ 
and that $L$ is a strong symmetric monoidal functor. 
This implies that the right adjoint $R$ is lax symmetric monoidal and both functors lift to the level of $\mcO$-algebras.
Note that a $\bV$-enriched adjunction lifts to algebras over compatible  $\bV$-operads.

The more general case of lifting adjunctions in case of
non-commutative monoids was
dealt with in \cite[Subsection 3.3]{SchwedeShipleyMon}.  These are all special cases of the adjoint lifting theorem, see for example \cite[Theorem 4.5.6]{borceux2}.

\begin{lemma}\label{lem:adjunctionlift}
Let $(L,R)$ be a $\bV$-enriched strong symmetric monoidal adjunction between $\bV$-tensored 
symmetric monoidal categories $\bC$ and $\bD$.

Let $\mcO$ be an operad in $\bV$,
then $L$ and $R$ extend to functors of $\mcO$-operads in $\bC$ and $\bD$, and we obtain 
a commutative square of adjunctions
\[
\xymatrix{
\algin{\mcO}{\bC} \ar@<0.5ex>[d]^U  \ar@<0.5ex>[r]^L
&
\algin{\mcO}{\bD} \ar@<0.5ex>[d]^U   \ar@<0.5ex>[l]^R
\\
\bC \ar@<0.5ex>[r]^L \ar@<0.5ex>[u]^{F_{\mcO}}
&
\bD . \ar@<0.5ex>[l]^R \ar@<0.5ex>[u]^{F_{\mcO}}
}
\]
\end{lemma}

\begin{remark} Notice, that we can consider a more general situation of lifting adjoint pairs to the level of algebras, when the left adjoint is \emph{not} strong symmetric monoidal. However, in our applications all left adjoints are strong symmetric monoidal, thus we restrict attention to this particular situation.
\end{remark}

It is automatic that the lifted adjunction of Lemma \ref{lem:adjunctionlift} 
forms a Quillen pair.

\begin{lemma}\label{lem:liftQP}
Suppose that $L$ and $R$ are a Quillen pair
and that the hypotheses of Lemma \ref{lem:liftedmodel}   hold so
that the categories of $\mcO$--algebras have lifted cofibrantly generated model structures and the
functors lift to an adjoint pair $L$, $R$.   Then  $L$, $R$ is a Quillen pair at the level of $\mcO$-algebras. \qed
\end{lemma}
\begin{proof} The result follows from the fact that the right adjoint at the level of algebras
is the same as at the level of underlying categories and fibrations and weak equivalences are created in the underlying categories.
\end{proof}

In the situation where Lemma \ref{lem:adjunctionlift} holds, 
we can look for criteria which will guarantee that a Quillen equivalence on the
original category lifts to a Quillen {\em equivalence} on
algebras. With $\mcO$-algebras replaced by non-commutative monoids, 
this is the content of  \cite[Theorem 3.12]{SchwedeShipleyMon}.
This lemma will be used to obtain most of the Quillen equivalences 
needed for the main result, Theorem \ref{thn:fullcompare}.

\begin{lemma}\label{lem:liftQE} Suppose that $L$ is a strong monoidal functor, $L$ and $R$ form a Quillen equivalence (at the level of categories $\bC$ and $\bD$) and that the hypotheses of Lemma \ref{lem:liftedmodel}   hold so
that the categories of $\mcO$--algebras have lifted cofibrantly generated model structures and the
functors lift to an adjoint pair $L$, $R$ at the level of $\mcO$--algebras.

 If  $U$ preserves cofibrant objects
then the lifted Quillen pair $L$ and $R$ at the level of $\mcO$-algebras is a Quillen
equivalence.
\end{lemma}
\begin{proof}
To show that $(L,R)$ is a Quillen equivalence at the level of $\mcO$--algebras we need to show that for a
cofibrant $X$ in $\mcOalgbbC$ and a fibrant $Y$ in $\mcOalgbbD$, $X\lra R Y$ is an
equivalence if and only if $L X\lra Y$ is an equivalence.

Since weak equivalences are created in the original categories,  the first condition
is that $UX \lra UR Y=RUY $ is an equivalence in $\bC$.

Now $UY$ is fibrant because the fibrations are created by $U$ and $UX$ is cofibrant by assumption. It follows from the fact that $(L, R)$ is a
Quillen equivalence at the level of underlying categories that $UX \lra RUY$ being a weak
equivalence is equivalent to $LUX \lra UY$ being a weak equivalence.

Finally, since $L U=UL$ the latter condition  is equivalent to $UL X\lra UY$
being a weak equivalence, which finishes the proof.
\end{proof}

\section{Lifting model structures for \texorpdfstring{$G$}{G}--spectra}\label{section:liftingModels}

In this section we lift various model structures on $G\endash \mathrm{Sp}$ to
the level of algebras over an operad $\mcO$ in $G$-spaces. The main example is the left Bousfield localisation of a positive stable model structure on $\GSp$ at a $G$--spectrum $E$, denoted by $L_E(\GSp)$. We establish compatibility conditions between $E$ and $\mcO$, which ensure that the right-lifted model structure on $\mcO$--algebras from $L_E(\GSp)$  exists.

Recall that a left Bousfield localisation of a positive stable model structure on $\GSp$ at a cofibrant object $E$ is a new model structure, where cofibrations are the positive stable cofibrations and a map $f$ is an $E$--equivalence if $E\wedge f$ is a stable weak equivalence.

We assume in this section that
$\mcO$ is an $\Ninfty$--operad, thus by Definition \ref{def:Ninfty_operad} $\mcO(n)$ is universal space for a family of
subgroups of $G \times \Sn$ that contains all subgroups of the form
$H \times \{e\}$. We call this family $\mcF_n(\mcO)$.
Two primary examples of such an operad are $\Einfty^G$ and $\Einfty^1$, see Example \ref{ex:operads}.

First, we consider compatibility conditions between a $G$--spectrum $E$ and a $G$--operad $\mcO$.

\begin{definition}\label{def:O-compatible}
We define the \textbf{geometric isotropy} of a $G$-spectrum $E$ as
\[
\Ig(E) = \{ H \leqslant G \mid \Phi^H E \not\simeq \ast \}.
\]
Let $p_G \co \GSn \to G$ be the projection.
We say that a $G$-spectrum $E$ is \textbf{$\mcO$-compatible}
if for all $n \geqslant 1$, whenever $\Gamma \in \mcF_n(\mcO)$
and $p_G \Gamma \in \Ig(E)$, then $\Gamma = p_G \Gamma \times \{ e \}$.
\end{definition}

\begin{example}\label{ex:compatible}
Notice that all $G$-spectra are compatible with an $\Einfty^1$ operad.
Any free $G$-spectrum $E$ is compatible with any $N_\infty$--operad, in particular with $\Einfty^G$.
\end{example}
\begin{example}
Let $G=C_6$ and take $\mcO:=\mathcal{D}(U)$, the little disc operad for the $G$--universe $U$ generated by $\bR[C_6/C_2]$ (see \cite[Definition 3.11 (ii)]{BlumbergHillNorms} for the definition of the little disc operad). We observe that $C_6$--spectrum $E:=\Sigma^\infty C_6/C_{2+}$ is $\mcO$--compatible.

First notice, that for $H \leqslant G$ by \cite[Theorem 4.19]{BlumbergHillNorms} the admissible $H$--sets $T$ for $\mcO$ are the ones such that there exists an $H$--equivariant embedding $T \longrightarrow U$.
Thus the only $C_2$--admissible sets are the trivial ones.

Recall from \cite[Proposition 4.2]{BlumbergHillNorms} that if $\Gamma\subset G\times \Sn$ is such that $\Gamma \cap (\{e\}\times \Sn)=\{e\}$, then there is a subgroup $H$ in $G$  and a homomorphism $f:H\longrightarrow \Sn$ such that $\Gamma$ is the graph of $f$. By \cite[Definition 4.3]{BlumbergHillNorms} an $H$--set $T$ is admissible for an operad $\mcO$ if the graph of the homomorphism $\Gamma_T: H \longrightarrow \Sigma_{|T|}$ defining $H$--structure on $T$ is in $\cF_{|T|}(\mcO)$.

Since the only $C_2$-admissible sets $T$ are the trivial ones, the only $\Gamma_T: C_2 \longrightarrow \Sigma_{|T|}$ give subgroups of the form $C_2 \times \{e\}$.

This implies the condition from the definition of $\mcO$--compatibility if $\Ig(E)=\{1,C_2\}$. Thus the $C_6$--spectrum $E:=\Sigma^\infty C_6/C_{2+}$ is $\mcO$--compatible.
\end{example}

Below we consider $G$-spectra with the model structure which is the left Bousfield localisation of the positive stable model structure at $E$, denoted by $L_E(G\endash \mathrm{Sp})$.
\begin{theorem}\label{thm:algebramodelstructure}
For $E$ an \textbf{$\mcO$-compatible} $G$-spectrum, there is a cofibrantly generated model structure on
$\mcO$-algebras in $G\endash \mathrm{Sp}$ where the weak equivalences
are those maps which forget to $E$-equivalences of $G$-spectra.
We call this the \textbf{$E$-local model structure on $\mcO$-algebras in
$G$-spectra}.
Furthermore the forgetful functor sends cofibrant algebras
to cofibrant spectra in $L_EG\endash \mathrm{Sp}$.
\end{theorem}
\begin{proof}
The argument is similar to that of \cite[Proposition A.1]{BlumbergHillNorms}
which is based on the general result of \cite[Proposition 5.13]{mmss01}.
This general result has two parts, the first part is about the interaction of pushouts,
sequential colimits and $h$-cofibrations (also known as the Cofibration Hypothesis).
The second part is that every map built using pushouts (cobase extensions) and
sequential colimits from $F_\mcO J_E$ (maps of the form $F_\mcO j$, for $j$ a generating acyclic
cofibration of $L_E( \GSp)$) should be $E$-equivalences.

The first part is well known and uses a standard technique of describing
pushouts of $\mcO$-algebras as sequential colimits in underlying category,
see Harper and Hess \cite[Proposition 5.10]{harperhess13} for details of this technique.

The second part can be broken into a number of steps.
\begin{description}
\item[Step 1] If $f \co X \to Y$ is a
cofibration of cofibrant $G$-spectra that is an $E$-equivalence, then
$F_\mcO f$ is an $E$-equivalence and a $h$-cofibration of underlying $G$-spectra.

\item[Step 2] Any pushout (cobase extension) of
$F_\mcO f$ is an $E$-equivalence and a $h$-cofibration of underlying $G$-spectra.

\item[Step 3] Any sequential colimit of such maps is an $E$-equivalence.
\end{description}

Step 1 follows from Lemma \ref{lem:relevant}.
Step 2 requires us to rewrite the pushout
in the previously--mentioned manner of Harper and Hess and check that
each stage of their construction preserves $E$-equivalences. This is standard
and uses Lemma \ref{lem:orbitsandE} to show that certain constructions
are $E$-equivalences.

Step 3 is routine. A sequential colimit of a collection of maps of $\mcO$-algebras (which forget to $h$-cofibrations)
is given by the sequential colimit in $G$-spectra (this is part of the Cofibration Hypothesis).
The result holds since smashing with $E$ commutes
with sequential colimits of $G$-spectra.
It follows that the model structure exists.

The statement about cofibrant objects follows from
\cite[Theorem 5.18]{harperhess13}. 
\end{proof}

By Example \ref{ex:compatible} we have two classes of model structures
where the $E$-local model structure on $\mcO$-algebras in $G$-spectra always exists.

\begin{corollary}\label{cor:einfmodelstructure}
For $E$ a $G$-spectrum, there is a cofibrantly generated model structure on
$\einf^1$-algebras in $G\endash \mathrm{Sp}$ where the weak equivalences
are those maps which forget to $E$-equivalences of $G$-spectra.
Furthermore, the forgetful functor preserves cofibrant objects.
\end{corollary}

\begin{corollary}
For $E$ a free $G$-spectrum, there is a cofibrantly generated model structure on
$\mcO$-algebras in $G\endash \mathrm{Sp}$ where the weak equivalences
are those maps which forget to $E$-equivalences of $G$-spectra.
Furthermore, the forgetful functor preserves cofibrant objects.
\end{corollary}

\begin{remark}\label{rmk:normsandsufficient}
Requiring that a spectrum $E$ is $\mcO$-compatible seems to ensure that
the algebra structures in $E$-local $G$--spectra encoded by the operad $\mcO$  have no non-trivial norms.
Since localisation at a general $E$ has an unpredictable effect on the norm maps,
it seems unreasonable to expect more with this approach.

For specific $E$, such as $S_\bQ$ in Appendix \ref{appendix:rational},
much more can be said. In particular, while compatibility is sufficient,
the appendix shows that it is not necessary.
\end{remark}

We may apply these results to some cases of interest to our
main result. For $H$ a subgroup of $G$, let $e_{(H)_G}$ be the idempotent of
the rational Burnside ring of $G$ supported by the
$G$-conjugacy class of $H$.
Let $N$ be the normaliser of $H$,
and $e_{(H)_N}$ the analogous idempotent to $e_{(H)_G}$
in the rational Burnside ring of $N$. We let $e_1$ be the
idempotent corresponding to the trivial subgroup of $W=N/H$.
The above results show that we have lifted model structures
on each of the following categories.
\[
\begin{array}{ll}
\algin{\einf^1}{  L_{e_{(H)_G}S_{\bQ}}(G\endash \mathrm{Sp})   }
&
\algin{\einf^1}{  L_{e_{(H)_N}S_{\bQ}}(N \endash \mathrm{Sp})   }
\\
\algin{\einf^1}{   L_{e_1 S_\bQ}(W \endash \mathrm{Sp})    }
&
\algin{\einf^1}{   \mathrm{Sp}_\bQ[W]   }
\end{array}
\]
Furthermore, in each case
forgetting the algebra structure preserves cofibrant objects
(even though cofibrations will not usually be preserved).

\section{Isotropic combinatorics}\label{sec:isotropic}

In this section we complete the proof of
Theorem \ref{thm:algebramodelstructure} by
proving Lemmas \ref{lem:orbitsandE} and \ref{lem:relevant}.
Throughout this section, our model structures will be based on
the positive stable model structure on $\GSp$,
so $L_E \GSp$ will be the left Bousfield localisation of the
positive stable model structure at $E$.

Assume that $f \co X\longrightarrow Y$ is a positive cofibration of positive cofibrant $G$-spectra that
is an $E$-equivalence. We want to show that
\[
F_{\mcO}^n f \co \mcO (n)_+\sm_{\Sn} X^{\sm n}\longrightarrow \mcO (n)_+\sm_{\Sn} Y^{\sm n}
\]
is a $E$-equivalence of $G$-spectra for each $n \in \bN$, provided that
$E$ is $\mcO$-compatible (see Definition \ref{def:O-compatible}).

Before we can prove Lemma \ref{lem:relevant} we need to study the
$\Sn$-orbits functor in some detail.
We start by relating smashing with $E$ to coequalisation over $\Sn$ and show that
\[
((\mcO (n)_+\sm X^{\sm n}) \smashprod E)/\Sn
\longrightarrow
((\mcO (n)_+\sm X^{\sm n} ) /\Sn) \smashprod E
=
(\mcO (n)_+\sm_{\Sn} X^{\sm n} ) \smashprod E
\]
is an isomorphism of $G$-spectra. While obvious at the space level,
the spectrum level construction depends upon the choice of universe for
$\GSnSp$.

\begin{remark}
Whenever we consider $\GSn$--spectra we index on a $\Sn$ fixed universe,
such as the inflation of a complete $G$-universe to a $\GSn$--universe.
The indexing spaces of such a universe are $G$-representations
considered as $\GSn$-representations via the projection
$p_G: \GSn \to G$.
It follows that a $G$-spectrum $B$, inflated to a $\GSn$--spectrum
is given by simply applying the space level inflation functor
$(p_G^{*} B)(V) = p_G^{*} (B(V))$
for $V$ a $G$-representation.
\end{remark}

\begin{lemma}\label{lem:orbitsandE}
Let $A$ be a $\GSn$-spectrum and $B$ a $G$-spectrum.
Then there is a natural isomorphism of $G$-spectra
\[
(A \smashprod B)/\Sn
\longrightarrow
(A/\Sn) \smashprod B.
\]
\end{lemma}
\begin{proof}
Let $U$ be an indexing space for $G$-spectra, then we can consider the smash product
before coequalising over $\bS$, denoted $\bar{\smashprod}$.
\[\begin{array}{rl}
\big( (A \bar{\smashprod} B)/\Sn \big) (U)
= & (A \smashprod B)(U) /\Sn  \\
= & \left( \int^{V, W} A(V) \smashprod
B(W) \smashprod \mcL(V \oplus W, U)   \right)/\Sn \\
= & \left( \int^{V, W} A(V)/\Sn \smashprod
B(W) \smashprod \mcL(V \oplus W, U)   \right) \\
= & (A/\Sn \bar{\smashprod} B)(U)
\end{array}
\]
This depends on the representations $V$, $W$ and $U$ all being
$\Sn$-fixed, so that we can move the $\Sn$-orbits functor past the
coend and the smash products of spaces. It follows that this isomorphism
passes to the actual smash product of orthogonal spectra.
\end{proof}

Whilst this lemma helps with the categorical nature of $\Sn$, we still need to consider
the homotopical properties of $(-)/\Sn$ and understand
when $(-)/\Sn$ preserve weak equivalences.
Recall the $\mcF$--model structure on $\GSnSp$ for $\mcF$ a family of subgroups of $\GSn$
from \cite[Section IV.6]{mm02}. Here the weak equivalences are given by taking $\Gamma$--fixed points
for each $\Gamma \in \mcF$ and are called $\mcF$--equivalences.
The cofibrations are the positive cofibrations
constructed from space level cofibrations
but where one uses only the homogeneous spaces $(\GSn)/\Gamma_+$ for $\Gamma \in \mcF$.
This model structure is denoted $\mcF \GSnSp$.

We use these definitions in the following lemma, which is given
so that we can later clarify why we need some relation between the spectrum $E$ and the operad
$\mcO$ in order to get a model structure $\algin{\mcO}{L_E \GSp}$, right-lifted from $L_E \GSp$.
We state the lemma for the positive model structures, but it also holds in the non-positive case.

\begin{lemma}\label{lem:quillenorbit}
Let $\mcF$ be a family of subgroups of $\GSnSp$, such that for any $H\leqslant G$, $H \times \{e\} \in \mcF$. Then there is a
Quillen pair
\[
\adjunction{(-)/\Sn}{\mcF\GSnSp}{\GSp}{{p_G}^*}.
\]
The right adjoint $p_G^*$ preserves positive cofibrations
and the left adjoint $(-)/\Sn$ sends $\mcF$--equivalences between
positive $\mcF$--cofibrant objects to weak equivalences of $G$--spectra.
\end{lemma}
\begin{proof}
This is a Quillen pair as $p_G^*$ preserves fibrations and acyclic fibrations
since (at the level of spaces) $(p_G^* A)^{\Gamma} = A^{p_G \Gamma}$.
The statement about $(-)/\Sn$ is then immediate from Ken Brown's Lemma.
It follows from the isomorphism
\[
p_G^* G/H_+ \cong (\GSn)/(H \times \{e\})_+
\]
and the assumption $H \times \{e\} \in \mcF$,
that $p_G^*$ preserves positive cofibrations.
\end{proof}

This lemma has some simple consequences. Consider a map
$f \co X \to Y$ of $\GSn$-spectra which are
(positive) cofibrant in the $\mcF$-model structure for some family $\mcF$.
Then to see that $f/\Sn$ is a weak equivalence in $\GSp$,
one only has to check the $\Gamma$--fixed points of $f$ are weak equivalences of spectra
for $\Gamma$ in $\mcF$.
Hence for a fixed $f$, the more subgroups are needed in $\mcF$ to make $X$ and $Y$ cofibrant
the more fixed points we need to check. In particular, if $\mcF= \{ H \times \{e\} \mid H \leqslant G \}$
then one only has to check the $H \times \{e\}$-fixed points for varying $H$. That is,
one only has to check that $f$ is a weak equivalence of $G$-spectra after forgetting the $\Sn$--action.

We can now reduce our problem substantially.
Let $f \co X \to Y$ be a positive cofibration of
positive cofibrant $G$--spectra and an $E$-equivalence. We want to show that
\[
F_{\mcO}^n f \co \mcO (n)_+\sm_{\Sn} X^{\sm n}\longrightarrow \mcO (n)_+\sm_{\Sn} Y^{\sm n}
\]
is a weak equivalence of $G$--spectra.
Lemmas \ref{lem:orbitsandE} and \ref{lem:quillenorbit} show that it suffices to prove
that
\[
\mcO (n)_+\sm f^{\sm n} \sm E:
\mcO (n)_+\sm X^{\sm n} \sm E
\longrightarrow
\mcO (n)_+\sm Y^{\sm n} \sm E
\]
is an $\mcF_n(\mcO)$--equivalence of $\mcF_n(\mcO)$--cofibrant $\GSn$ spectra.
The filtration of Harper and Hess \cite[Proposition 5.10]{harperhess13} shows that
$X^{\sm n}$ is a positive cofibrant $\GSn$-spectrum.
Lemma \ref{lem:quillenorbit} shows that $E$,
inflated to a $\GSn$-spectrum is also cofibrant.
Hence \cite[Lemma IV.6.6]{mm02}
tells us that $\mcO (n)_+ \sm X^{\sm n} \smashprod E$ is
an $\mcF_n(\mcO)$--cofibrant $\GSn$ spectrum.
To see that $F_{\mcO}^n f$ is an $E$-equivalence,
we must check that $\mcO (n)_+\sm f^{\sm n} \sm E$ is an
$\mcF_n(\mcO)$--equivalence. In fact, this is equivalent to checking that
$\mcO (n)_+\sm f^{\sm n} \sm E$ is a weak equivalence of $\GSnSp$ by
\cite[Proposition IV.6.7]{mm02}.
To show this last statement holds,
we use geometric fixed points and see how the
geometric isotropy subgroups of $E$ interact with $\mcF_n(\mcO)$.

\begin{definition}
Let us say that a subgroup $\Gamma \leqslant \GSn$ is
\textbf{irrelevant} if the $\Phi^\Gamma$-fixed
points of $\mcO(n)_+$ or $p_G^*E$ are trivial and \textbf{relevant} otherwise.
\end{definition}
If $E$ is $\mcO$--compatible, then the only relevant subgroups are of the form
$H \times \{ e \}$.
Let $p_G \co \GSn \to G$ be the projection, then
for $A \leqslant \GSn$, the $A$-geometric fixed points of $E$ inflated to a $\GSn$ spectrum
are given by $\Phi^{p_G(A)} E$. The next lemma is then immediate.

\begin{lemma}
The relevant subgroups for $\mcO (n)_+$ and $E$  lie in
\[
\pushQED{\qed}
\mcF_n(\mcO)\cap \GI_{\GSn}(p_G^*E)=\{ A \st A\in \mcF_n(\mcO), \ p_G(A)\in
\GI_G(E)\}.
\qedhere
\popQED
\]
\end{lemma}

\begin{lemma}\label{lem:relevant}
Let $f \co X \to Y$ be a positive cofibration between
positive cofibrant $G$--spectra and an $E$-equivalence.
Assume that $E$ is $\mcO$--compatible, then
\[
\mcO (n)_+\sm f^{\sm n} : \mcO (n)_+\sm X^{\sm n} \longrightarrow \mcO (n)_+\sm Y^{\sm n}
\]
is an $E$-equivalence.
\end{lemma}

\begin{proof}
We only need to consider the relevant subgroups,
as both sides are trivial at an irrelevant subgroup.
By the compatibility condition, any relevant subgroup is of the form 
 $H \times \{ e \}$.
If we view $X^{\sm n}$ just as a $H \times \{ e \}$-spectrum then $X^{\sm n}$ is
just a smash product of copies of $X$ and we have
$\Phi^{H \times \{e\}}(X^{\sm  n})=
(\Phi^H X)^{\sm n}$. Hence
$$\Phi^{H\times \{ e \} }(\mcO (n)_+\sm X^{\sm n}\sm E) \simeq \Phi^{H\times
  \{ e \} }(\mcO (n)_+)\sm (\Phi^HX)^{\sm n}\sm \Phi^HE. $$
If $f: X \longrightarrow Y$ is an $E$-equivalence, $\Phi^HX\longrightarrow \Phi^HY$ is an
$\Phi^H E$-equivalence. Moreover $\Phi^H$ preserves cofibrations
and the smash product of $E'$-equivalences of cofibrant objects
is an $E'$-equivalence for any suitable spectrum $E'$, such as $\Phi^H E$.
\end{proof}

\section{Lifting the Quillen equivalences}\label{sec:liftingQE}

In this section we construct our sequence of Quillen equivalences between $E_\infty^1$-algebras in
rational $G$-spectra and commutative algebras in the algebraic model. The strategy is to replace
$E_\infty^1$ by $\Comm$ as soon as we have moved from
$L_{e_{(H)_G}S_{\bQ}}(G\endash \mathrm{Sp})$ to $W_G H$-objects in spectra. We then apply the results of
\cite{richtershipley} to compare commutative algebras in rational spectra to
commutative algebras in rational chain complexes.

Before we check the compatibility of the zig-zag of Quillen equivalences from \cite{KedziorekExceptional} with algebras over $\Einfty^1$-algebras we note that this zig-zag of Quillen equivalences works for positive stable model structures.

\begin{remark} The model structures and Quillen equivalences of \cite{KedziorekExceptional}
relating rational $G$-spectra to the algebraic model can all be obtained using the
positive stable model structure on the various categories of spectra.
Moreover, the adjunction
\[
\adjunction{H\bQ \wedge -}{Sp_\bQ^\Sigma[W] }{(H\bQ \leftmod)[W]}{U}
\]
is a Quillen equivalence when the left side is considered with the positive stable model structure and the right hand side is considered with the positive flat stable model structure of \cite{Shipley_convenient}. This model structure on $(H\bQ \leftmod)[W]$ is the starting point for the work in \cite{richtershipley}.
\end{remark}

The results of Section \ref{section:liftingModels} allow us to apply
Lemmas \ref{lem:adjunctionlift} and \ref{lem:liftQE}
to obtain Quillen equivalences
\[
\xymatrix@R=2pc{
\algin{\einf^1}{  L_{e_{(H)_G}S_{\bQ}}(G\endash \mathrm{Sp})   }
\ar@<-1ex>[d]_{i^{*}}
\\
\algin{\einf^1}{  L_{e_{(H)_N}S_{\bQ}}(N \endash \mathrm{Sp})   }
\ar@<-1ex>[u]_{F_N(G_+,-)}
\ar@<+1ex>[d]^{(-)^H}
\\
\algin{\einf^1}{   L_{e_1S_\bQ}(W_G H \endash \mathrm{Sp})    }
\ar@<+1ex>[u]^{ \epsilon^\ast}
\ar@<+1ex>[d]^{\mathrm{res}}
\\
\algin{\einf^1}{   \mathrm{Sp}_\bQ[W_G H]  . }
\ar@<+1ex>[u]^{\mathrm{I}_t^c}
}
\]

Now we change the operad from the $\einf^1$ to $\Comm$ obtaining a Quillen equivalence.
\begin{lemma}\label{lem:einftocomm}
There is an adjunction
\[
\adjunction{\eta_*}
{\algin{\einf^1}{   \spO_\bQ[W]   }}
{\algin{\Comm}{   \spO_\bQ[W]   }}
{\eta^*}
\]
induced by the map of operads in $\Top_*$, $\eta: \einf^1 \lra \Comm$. This adjunction is a Quillen equivalence with respect to right-induced model structures from the right induced model structure on $W$-objects from the rational positive stable model structure on $\spO_\bQ$.
\end{lemma}
\begin{proof}
We first prove the result in the case $W$ is the trivial group.
The model structure on the right exists by Theorem \ref{thm:commGsp}
and the model structure on the left exists by Corollary \ref{cor:einfmodelstructure}
(here $G$ is the trivial group).
Since both model structures are lifted and $\eta^*$ forgets to the identity
functor of $\spO_\bQ$, it follows that the right adjoint
is a right Quillen functor that preserves all weak equivalences.
The derived unit is the same as the derived unit of the non-rational
adjunction and hence is a $\pi_*$-isomorphism.

Now we consider non-trivial $W$.
Let $\bC$ be a symmetric monoidal category, $W$ a finite group
and $\mcO$ an operad such that the category of $\mcO$--algebras in $\bC$ exists.
It is easy to see that the category of $\mcO$-algebras in $\bC$ with $W$--action, $(\algin{\mcO}{\bC})[W]$
is isomorphic to the category of $\mcO$--algebras in $\mcC[W]$,  $\algin{\mcO}{\bC[W]}$.
Extending to categories of $W$-objects preserves Quillen equivalences by
\cite[Proposition 5.4]{KedziorekExceptional}, so the general result follows immediately.
\end{proof}

\begin{lemma}
For $W$ a finite group, there is a Quillen equivalence
\[
\adjunction{\mathbb{P} \circ |-|}
{\algin{\Comm}{   \spS_\bQ[W]    }}
{\algin{\Comm}{   \mathrm{Sp}_\bQ[W]   }}
{\mathrm{Sing}\circ \mathbb{U}}
\]
where both model structures are right induced from the rational positive stable model structures.
\end{lemma}
\begin{proof}
This adjunction satisfies the conditions of Lemma \ref{lem:liftQE}, hence we can lift the
equivalence between
$Sp_\bQ^\Sigma[W] $ and $\mathrm{Sp}_\bQ[W]$ to the level of $\Comm$--algebras.
\end{proof}

\begin{proposition}
For $W$ a finite group, there is a Quillen equivalence
\[
\adjunction{H\bQ\wedge -}
{\algin{\Comm}{   \spS_\bQ[W]   }}
{\algin{\Comm}{   (H\bQ \leftmod)[W]  }}.
{U}
\]
\end{proposition}
\begin{proof}
The right hand model structure exists by \cite[Theorem 7.16]{harperhess13}.
This adjunction exists by Lemma \ref{lem:adjunctionlift}.
It satisfies the conditions of Lemma \ref{lem:liftQE}, hence we can lift the Quillen equivalence between
$\spS_\bQ[W]$ and $(H\bQ\leftmod)[W]$ to the level of $\Comm$--algebras.
\end{proof}

The paper \cite{richtershipley} of Richter and Shipley proves that
the model category of commutative $H \bR$-algebras (in symmetric spectra)
is Quillen equivalent to $\einf$-algebras in unbounded chain complexes of $R$-modules.
When $R=\bQ$ this gives the following

\begin{theorem}[{\cite[Corollary 8.4]{richtershipley}}]
There is a zig-zag of Quillen equivalences between the model category of commutative
$H \bQ$--algebras (in symmetric spectra) and differential graded commutative $\bQ$--algebras.
\end{theorem}

The model structure on commutative $H \bQ$--algebras is lifted from the positive
stable model structure on symmetric spectra, see for example \cite{Shipley_convenient} or \cite{hss00}. The model structure on
differential graded commutative $\bQ$--algebras has fibrations the surjections
and weak equivalences the homology isomorphisms.

\begin{corollary}\label{cor:operadshipleyfy}
Let $W$ be a finite group.
There is a zig-zag of Quillen equivalences between the model category of commutative
$H \bQ$--algebras with $W$--action
and differential graded commutative $\bQ$--algebras with a $W$--action.
\end{corollary}

\begin{theorem}\label{thm:splitcompare}
There is a zig-zag of Quillen equivalences 
\[
\algin{\einf^1}{  L_{e_{(H)_G}S_{\bQ}}(\GSp)   }
\  \simeq \ 
\algin{\Comm}{\Ch(\bQ[W_G H])}.
\]
\end{theorem}
\begin{proof}
The above results give a zig-zag of Quillen equivalences 
\[
\algin{\einf^1}{  L_{e_{(H)_G}S_{\bQ}}(\GSp)   }
\  \simeq \ 
\algin{\Comm}{   (H\bQ\leftmod)[W_G H]  }
\]
Corollary \ref{cor:operadshipleyfy} shows that
$\algin{\Comm}{   (H\bQ\leftmod)[W_G H]  }$ is Quillen equivalent to the model category
$\algin{\Comm}{\Ch(\bQ[W_G H])}$ so the result holds.
\end{proof}

We collect all these Quillen equivalences into one result.
\begin{theorem}\label{thn:fullcompare}
There is a zig-zag of Quillen equivalences
\[
\algin{\einf^1}{  (\GSp)_\bQ   }
\ \simeq \ 
\algin{\Comm} {\prod_{(H) \leqslant G} \Ch(\bQ[W_G H])}.
\]
\end{theorem}
\begin{proof}
It is clear that
\[
\algin{\Comm} {\prod_{(H) \leqslant G} \Ch(\bQ[W_G H])}
=
\prod_{(H) \leqslant G} \algin{\Comm} {\Ch(\bQ[W_G H])}.
\]
Whereas on the topological side,
the symmetric monoidal Quillen equivalence
\[
\adjunction
{\Delta}
{(\GSp)_\bQ}
{\prod_{(H) \leqslant G}  L_{e_{(H)_G}S_{\bQ}}(\GSp)  }
{\prod}
\]
of \cite[Corollary 6.4]{barnesfinite}
lifts to the level of $\einf^1$--algebras by Lemma \ref{lem:liftQE}.
Since taking $\einf^1$--algebras commutes with the product of categories, we may then use
Theorem \ref{thm:splitcompare} on each factor to obtain the
desired zig-zag of Quillen equivalences.
\end{proof}

\begin{remark}\label{rem:different_cats}
Each of the following categories has a model structure lifted from
the corresponding category of spectra.
\begin{center}
\begin{tabular}{p{5cm}p{5cm}}
$\algin{\einf^1}{   L_{e_1S_\bQ}(W\endash \spO)    }$
&
$\algin{\Comm}{   L_{e_1S_\bQ}(W \endash \spO)    }$ \\
${\algin{\einf^1}{  \spO_\bQ^\Sigma[W]    }}$
&
${\algin{\Comm}{   \spO_\bQ^\Sigma[W]    }}$ \\
$\algin{\einf^1}{   \spO_\bQ[W]   }$
&
$\algin{\Comm}{   \spO_\bQ[W]   } $\\
$\algin{\einf^1}{   (H\bQ \leftmod)[W]  }$
&
$\algin{\Comm}{   (H\bQ\leftmod)[W]  }$ \\
\end{tabular}
\end{center}

Since each of the model structures are lifted,
the model categories of the left hand column are all Quillen equivalent and
similarly so for the right hand column.
Lemma \ref{lem:einftocomm} gives a Quillen equivalence between the two columns
(at the level of $\spO_\bQ[W]$), hence all of the above model categories are Quillen equivalent.
In particular, we have several routes to obtain the Quillen equivalence of
Theorem \ref{thn:fullcompare}, but they all give the same derived equivalence.

We cannot, however, put a lifted model structure on the two model categories
\begin{center}
\begin{tabular}{p{5cm}p{5cm}}
$\algin{\Comm}{L_{e_{(H)_G}S_{\bQ}}(\GSp)}$ 
&
$\algin{\Comm}{L_{e_{(H)_N}S_{\bQ}}(\NSp)}$ \\
\end{tabular}
\end{center}
The difficulty is on the existence of multiplicative norms, 
as we discuss in Section \ref{sec:obstructions}. 
\end{remark}

\section{Norm type obstructions}\label{sec:obstructions}

The first step in the zig-zag of Quillen equivalences used to provide an algebraic model for rational $G$-spectra, when $G$ is a finite group is the splitting result of Barnes \cite{barnessplitting}. It uses the  idempotents of the rational Burnside ring corresponding to conjugacy classes of all subgroups of $G$. This splitting is precisely the step which fails to preserve any more ($G$-commutative) structure on algebras than the one for $E_\infty^1$-algebras. This is closely related to the fact that at the homotopy level any algebra with more structure will have non-trivial norm maps and the complete idempotent splitting in the first step of the zig-zag kills any non-trivial norm maps. This shows, that a Quillen equivalence of $E_\infty^1$-algebras and commutative algebras in the algebraic model is the most one could hope for using this approach. Let us analyse the situation in more detail below.

First recall that if $X$ is a $G$-spectrum 
which is non-equivariantly contractible, such that 
 $\pi_0^H(X)$ is a unital ring for all $H \leqslant G$ 
and $X$ has norms then $X\simeq *$.
By having norms we mean that there is a map $N_1^H:
\pi_0^1(X)\lra \pi_0^H(X)$ which is compatible with the product in the
sense that $N_1^H(0)=0$, $N_1^H(1)=1$ and
$N_1^H(xy)=N_1^H(x)N_1^H(y)$. Now if  $\pi_0^1(X)=0$ then $1=0$, and  it follows that $1=0$ in $\pi_0^H(X)$ so that $\pi^H_0(X)=0$. Hence $\pi^H_*(X)=0$ for all $H$ and $X\simeq *$.

\begin{remark}The argument above implies that there is no right-lifted model structure on $\einf^\cF$-algebras from $L_{e_{(H)_G}S_{\bQ}}(G\endash \mathrm{Sp})$, when $H\neq 1$ and $\cF\neq 1$. Suppose there is one. Then for any $\einf^\cF$-algebra $A$ there is a fibrant replacement (in $\einf^\cF$-algebras) $A^f$, which is underlying $e_{(H)_G}$--local and thus by the argument above it is equivariantly contractible. Thus any $A$ is weakly equivalent to a point in the model structure lifted from the $e_{(H)_G}$--local model structure on $G$--spectra. However, not every  $\einf^\cF$-algebra in $G$--spectra is $e_{(H)_G}$--locally equivalent to a point (for example the sphere spectrum). Since the forgetful functor was supposed to create weak equivalences at the level of  $\einf^\cF$-algebras, we get a contradiction. A similar argument shows that there is
no right-lifted model structure on $\Comm$-algebras from 
$L_{e_{(H)_G}S_{\bQ}}(G\endash \mathrm{Sp})$.
\end{remark}

\appendix
\section{Rationalisation of commutative rings}\label{appendix:rational}

In this section we show that there is a right-lifted model structure on commutative algebras in rational $G$--spectra, where $G$ is a compact Lie group. In particular, this argument shows that there is a right-lifted model structure on the category of commutative algebras in rational \emph{non--equivariant} spectra, used in several lemmas in Section \ref{sec:liftingQE}. This seems to be a folklore result, however the authors do not know of a reference for it in the literature.

Fix $E=S_\bQ$, the non-equivariant rational sphere spectrum
that we inflate to the level of $G$-spectra whenever needed.
The spectrum $S_\bQ$ is constructed as the homotopy colimit of the sequence
\[
S^0 \overset{2}{\lra}
S^0 \overset{3}{\lra}
S^0 \overset{4}{\lra}
S^0 \overset{5}{\lra} \dots
\]
in $\ho (\spO)$, hence it is a cofibrant spectrum (and can be made positive
cofibrant by replacing each $S^0$ by its positive cofibrant replacement before taking homotopy colimits).
It is weakly equivalent to the non-equivariant
Eilenberg--MacLane spectrum $H \bQ$, which is a commutative ring spectrum.

We want to show that free algebra functor for $\Comm$ takes rational equivalences
(that are cofibrations between cofibrant objects in the rational positive stable model structure) to rational equivalences.
We have to use a different approach to the one 
in Section \ref{sec:isotropic},
as $\Comm$ uses the family
of all subgroups of $\GSn$ and is not an $\Ninfty$--operad. 
The key fact that we will use is that
$S_\bQ$ has a $\Sn$-equivariant `multiplication' map
$S_\bQ^{\smashprod_n} \lra S_\bQ$ that is a
weak equivalence of $\Sn$-spectra. This map is
modelling the multiplication of $H \bQ$.

\begin{lemma}\label{lem:mult}
There is a $\Sn$-equivariant
map $S_\bQ^{\smashprod_n} \lra S_\bQ$ of $\Sn$-spectra
that is a $\pi_*$-isomorphism on $\Sn$-spectra (indexed on a trivial universe).

Furthermore, if we inflate this to $\GSnSp$ then it
becomes a weak equivalence of  $\GSnSp$.
\end{lemma}
\begin{proof}
Consider the commutative diagram in $\ho (\SnSp)$:
\[
\xymatrix{
(S^0)^{\sm n} \ar[d]^\cong \ar[r]^{2^{\sm n}} &
(S^0)^{\sm n} \ar[d]^\cong \ar[r]^{3^{\sm n}} &
(S^0)^{\sm n} \ar[d]^\cong \ar[r]^{4^{\sm n}} &
(S^0)^{\sm n} \ar[d]^\cong \ar[r]^{5^{\sm n}} &
\dots \\
S^0 \ar[r]^{2^n} &
S^0 \ar[r]^{3^n} &
S^0 \ar[r]^{4^n} &
S^0 \ar[r]^{5^n} &
\dots
}
\]
the vertical maps are the commutative multiplication operation of $S^0$
and hence are $\Sn$--equivariant.
Taking homotopy colimits we obtain the desired weak equivalence of $\SnSp$:
$S_\bQ^{\smashprod_n} \lra S_\bQ$.

The second statement follows since for any $\Sn$-spectrum $A$
$(p_{\Sn}^* A)^\Gamma = A^{p_{\Sn} \Gamma}$,
where $p_{\Sn} \Gamma$ is the image of $\Gamma$ under the projection to $\Sn$.
\end{proof}

\begin{theorem}\label{thm:commGsp}
There is a cofibrantly generated model structure on
$\Comm$-algebras in $G\endash \mathrm{Sp}$ where the weak equivalences
are those maps which forget to rational equivalences of $G$-spectra.
\end{theorem}
\begin{proof}
The proof is very similar to the case of a general operad.
The main changes are in the three steps used to show that
every map built using pushouts (cobase extensions) and
sequential colimits from $F_\mcO J_\bQ$ (maps of the form $F_\mcO j$, for $j$ a generating
cofibration of the rational positive stable model structure on $\GSp$) is a rational equivalence.

\begin{description}
\item[Step 1] If $f \co X \to Y$ is a
cofibration of cofibrant $G$-spectra that is a rational equivalence, then
$F_\Comm f$ is a rational equivalence and an $h$-cofibration of underlying $G$-spectra.

\item[Step 2] Any pushout (cobase extension) of
$F_\Comm f$ is a rational equivalence and an $h$-cofibration of underlying $G$-spectra.

\item[Step 3] Any sequential colimit of such maps is a rational equivalence.
\end{description}

The first step is a difficult one, we start by comparing $\Comm$ to the $E_\infty^G$
(this is the same method as is used to show that there is
a model structure on $\Comm$-algebras in $\GSp$).
If $X$ is a cofibrant $G$-spectrum then the map
\[
(E_G \Sn)_+ \sm_{\Sn} X^{\smashprod_n}  \lra
X^{\smashprod_n}/\Sn
\]
is a (natural) weak equivalence of $\GSn$--spectra
by \cite[Proposition B.117]{HillHopkinsRavenel}
or \cite[Lemma 8.4]{mm02}.
Here $(E_G \Sn)_+$ is the universal space for the family $\mcF_n^G$.
Hence it suffices to show that if $f \co X \to Y$ is a cofibration of cofibrant
$G$-spectra that is a rational equivalence, then
\[
(E_G \Sn)_+ \sm_{\Sn} f^{\smashprod_n} \co
(E_G \Sn)_+ \sm_{\Sn} X^{\smashprod_n} \lra
(E_G \Sn)_+ \sm_{\Sn} Y^{\smashprod_n}.
\]
is a rational equivalence.

By Lemma \ref{lem:mult} there is a $\GSn$-equivariant
map $S_\bQ^{\smashprod_n} \to S_\bQ$ that is a
weak equivalence of $\GSn$-spectra.
We use this to obtain the following series of maps of
$G$-spectra
\[
\begin{array}{rcl}
((E_G \Sn)_+ \sm_{\Sn} X^{\smashprod_n}) \smashprod S_\bQ
& \overset{\cong}{\lra} &
((E_G \Sn)_+ \sm X^{\smashprod_n} \smashprod S_\bQ)/\Sn \\
& \longleftarrow &
((E_G \Sn)_+ \sm X^{\smashprod_n} \smashprod S_\bQ^{\smashprod_n})/\Sn \\
& \overset{\cong}{\lra} &
((E_G \Sn)_+ \sm (X \smashprod S_\bQ)^{\smashprod_n})/\Sn. \\
\end{array}
\]
The second map is a weak equivalence by Lemma \ref{lem:quillenorbit},
hence we have a zig-zag of weak equivalences of $G$-spectra.
Since $X \to Y$ is a rational equivalence, it follows that
\[
((E_G \Sn)_+ \sm (X \smashprod S_\bQ)^{\smashprod_n})/\Sn
\lra
((E_G \Sn)_+ \sm (Y \smashprod S_\bQ)^{\smashprod_n})/\Sn
\]
is a  weak equivalence of $G$-spectra.
Hence the map
$X^{\smashprod_n}/\Sn \to Y^{\smashprod_n}/\Sn$
is a rational equivalence of $G$-spectra.

The remaining steps hold just as before.
\end{proof}

The same strategy gives a related result for $\Ninfty$-operads in place of~$\Comm$.  
\begin{corollary}
For an $\Ninfty$--operad $\mcO$, 
there is a cofibrantly generated model structure on
$\mcO$-algebras in $G\endash \mathrm{Sp}$ where the weak equivalences
are those maps which forget to rational equivalences of $G$-spectra.
\end{corollary}

\end{document}